\documentclass{amsart}
\usepackage{mathrsfs}
\usepackage[all]{xy}

\usepackage{wasysym}
\usepackage{amssymb}
\usepackage{comment}

\usepackage{ifpdf}
\ifpdf %if using pdfLaTeX in PDF mode
  \usepackage[pdftex]{graphicx}
  \DeclareGraphicsExtensions{.pdf,.png,.jpg,.jpeg,.mps}
  \usepackage{pgf}
\else %if using LaTeX or pdfLaTeX in DVI mode
  \usepackage{graphicx}
  \DeclareGraphicsExtensions{.eps,.bmp}
  \usepackage{pgf}
  \usepackage{pstricks}
\fi
\usepackage{epic,bez123}
\usepackage{wrapfig}% package for wrapfigure environment

\newtheorem{thm}{Theorem}[section]

\newtheorem{cor}[thm]{Corollary}
\newtheorem{lem}[thm]{Lemma}

\newtheorem*{claim}{Claim}

\theoremstyle{definition}
\newtheorem{defn}[thm]{Definition}

\theoremstyle{remark}

\newtheorem*{rem}{Remark}
\newtheorem{example}[thm]{Examples}

\newtheorem{constants}[thm]{Constants}
\newtheorem*{ack}{Acknowledgment}

\newcommand{\Gx }{\mathscr{G} (G, S)}

\newcommand{\GP }{(G, \mathcal P)}

\newcommand{\pG}{\partial{G}}

\newcommand{\PS}[1]{\mathcal P(s,{#1})}

\newcommand{\g}[1]{\nu_{#1}}

\newcommand{\diam }[1]{\textbf{Diam}(#1)}

\newcommand{\proj }{\textbf{Proj}}

\newcommand{\len }{\ell}

\usepackage[bookmarks=true, pdfauthor={YANG Wenyuan}]{hyperref}

\begin{document}

\title{Statistical hyperbolicity of relatively hyperbolic groups}

\author{Jeremy Osborne}

\address{Kenosha, WI 53144, United States}
\email{Dr.J.A.Osborne@gmail.com}
 
\author{Wen-yuan Yang}

\address{Beijing International Center for Mathematical Research,
Peking University, Beijing, 100871, P.R.China}

%\curraddr{Le Département de Mathématiques de la Faculté des
%Sciences d'Orsay, Université Paris-Sud 11, France}
\email{wyang@math.pku.edu.cn}
\thanks{}

% \author{Author Two}
% \address{Mathematical Research Section, School of Mathematical Sciences,
% Australian National University, Canberra ACT 2601, Australia}
% \email{two@maths.univ.edu.au}
% \thanks{Support information for the second author.}

%    General info

\date{\today}

\dedicatory{}


\begin{abstract}
We prove that a non-elementary relatively hyperbolic group is statistically hyperbolic with respect to every finite generating set. We also establish the statistical hyperbolicity for certain direct products of two groups, one of which is relatively hyperbolic.

%Statistical hyperbolicity for a group action or a group with left-invariant proper metric.
\end{abstract}

\maketitle

\section{Introduction}
%Since the introduction of hyperbolic groups by Gromov in 1987, it has been a central theme to import non-positive curvatures from smooth setting into the study of group theory, which forms the field of geometric group theory.  

The idea of statistical hyperbolicity was first introduced by M. Duchin, S. Leli\`evre, and C. Mooney in \cite{DLM}.  Let $G$ be a group generated by a finite set $S$. Assume that $1 \notin S=S^{-1}$. Denote by $\Gx$ the Cayley graph of $G$ with respect to $S$. Consider the natural combinatorial metric on $\Gx$, denoted by $d$, inducing a word metric on $G$. The intuitive meaning of statistical hyperbolicity of a group can then be summed up as follows: On average, random pairs of points x,y on a sphere of the Cayley graph of the group almost always have the property that $d(x,y)$ is nearly equal to $d(x, 1) + d(1, y)$. More precisely,

\begin{defn}
Denote $S_n=\{g\in G: d(1, g) = n\}$ for $n\ge 0$. Define
$$
E(G, S) =\lim_{n\to \infty} \frac{1}{|S_n|^2} \sum_{x, y\in S_n} \frac{d(x, y)}{n},
$$
if the limit exists.
The pair $(G, S)$ is called \textit{statistically hyperbolic} if $E(G, S)=2$.
\end{defn}

Recall that a group is called \textit{elementary} if it is a finite group or a finite extension of $\mathbb Z$. It is easily checked that an elementary group is not statistically hyperbolic with respect to any generating set. In \cite{DLM}, Duchin-Leli\`evre-Mooney proved that $\mathbb Z^d$ for $d\ge 2$ is not statistically hyperbolic for any finite generating set. It was also discovered by Duchin-Mooney in \cite{DM} that the integer Heisenberg group with any finite generating set is not statistically hyperbolic.

A list of statistically hyperbolic examples were also found in \cite{DLM}:
\begin{example}\label{stahypexamples}
\begin{enumerate}
\item
Non-elementary hyperbolic groups for any finite generating set.
\item
Direct product of a non-elementary hyperbolic group and a group for certain finite generating sets.
\item
The lamplighter groups $\mathbb Z_m \wr \mathbb Z$ where $m \ge 2$ for certain
generating sets.
\end{enumerate}
\end{example}
We remark that an analogous definition of statistical hyperbolicity to the above can be considered for any metric space with a measure. (Here for graphs we consider the counting measures). We refer the reader to \cite{DLM} for precise definitions. For any $m, p \ge 2$, the Diestel-Leader graph $DL(m, p)$ is proved to be statistically hyperbolic in \cite{DLM}. In \cite{DDM}, Dowdall-Duchin-Masur established the statistical hyperbolicity for Teichm$\ddot{u}$ller space with various measures.

Summarizing the above results, one could think of the number $E(G, S)$ as a measurement of negative curvature in groups and spaces. So it would be natural to expect that the statistical hyperbolic property holds  for a more general class of groups with negative curvature. A natural source of such groups to be investigated  is the class of relatively hyperbolic groups, which generalizes word hyperbolic groups, and includes many more examples such as
\begin{enumerate}
\item
fundamental groups of non-uniform lattices with negative curvature
\cite{Bow3}, 
\item
free products of groups, or a finite graph of groups with finite edge groups,
\item
limit groups \cite{Dah}, and 
\item
CAT(0) groups with
isolated flats \cite{HruKle}.
\end{enumerate}
We refer the reader to Section 2 and references therein for more details on relatively hyperbolic groups. The purpose of this article is then to generalize the first two items in Examples \ref{stahypexamples} to the setting of relatively hyperbolic groups.

Recently, the first-named author has established in his thesis \cite{Osborne} that  relatively hyperbolic groups are statistically hyperbolic, provided that the group growth rate dominates the ones of parabolic subgroups.   Our first result is to drop this assumption and to establish the full generalization of Duchin-Leli\`evre-Mooney's above result in relatively hyperbolic groups. 
\begin{thm}\label{mainthm}
A non-elementary relatively hyperbolic group is statistically hyperbolic with respect to every finite generating set.
\end{thm}

Let's say a bit about the ingredients in proof of our Theorem.  It was observed in \cite{DLM} that statistical hyperbolicity appears to be more delicate than the usual metric notion of hyperbolicity in the sense of Gromov. Namely, examples of trees can be produced to have arbitrary number $E(G, S) \in [0, 2]$. These examples lack of homogenenety cannot afford many isometries. Thus in their proof of statistical hyperbolicity for hyperbolic groups, Duchin-Leli\`evre-Mooney make essential use of a result of Coorneart about growth function in \cite{Coor}. This is recently generalized by the second-named author in \cite{YANG7} for relatively hyperbolic groups, cf. Lemma \ref{expball}.  Apart from this, we also exploit a crucial fact in \cite{YANG7} to obtain the full generality: parabolic groups have convergent Poincar\`e series (Corollary \ref{convergent}). Based on them, our proof follows roughly the outline in hyperbolic case but with more involved analysis. 
 
We now state our second result about direct product of two groups, one of which is relatively hyperbolic.  First recall the notion of \textit{growth rate} $\nu_{G, S}$ of a group $G$ relative to $S$, which is defined to be the limit
$$
\nu_{G, S} = \lim\limits_{n \to \infty} \frac{\log |S_n|}{n}.
$$ 
A generating set $S$ for $G\times H$ is called \textit{split} if every generator in $S$ lies either in $G$ or in $H$.  Denote $S_G:=S\cap G$ and $S_H:=S\cap H$. Taking into account Theorem \ref{mainthm}, we obtain the following theorem extending a similar result in \cite{DLM}.
\begin{thm}\label{directproduct}
Let $G\times H$ be a direct product of a non-elementary relatively hyperbolic group $G$ and a group $H$. Let $S$ be a split finite generating set for $G\times H$ and $S_G, S_H$ be the corresponding generating sets for $G, H$. If $\nu_{G, S_G} > \nu_{H, S_H}$, then $(G\times H, S)$ is statistically hyperbolic.
\end{thm}

It is obvious that Theorem \ref{directproduct} can be thought of as a generalization of Theorem \ref{mainthm}. 

At last, we further derive the following corollary from Theorem \ref{directproduct}. Recall that a group is called of \textit{sub-exponential growth} if its growth rate is zero for some (thus any) generating set. It is well-known that a non-elementary relatively hyperbolic group has exponential growth. 
\begin{cor}\label{directproductcor}
A direct product of a non-elementary relatively hyperbolic group and a group of sub-exponential growth is statistically hyperbolic with respect to  finite split generating sets.
\end{cor}

This article is structured as follows. Section 2 prepares preliminary material to be used in the proof of Theorem \ref{mainthm}, which occupies the whole Section 3. In Section 4, we give a proof of Theorem \ref{directproduct}.

\begin{comment}
\begin{cor}[??]
The Baumslag-Solitar groups $BS(m, n)$ for $|m|\ne |n|$ are statistically hyperbolic with respect to the natural generating set.
\end{cor}

Could we derive any consequence of a group with statistically hyperbolicity?
What about the semi-direct product? Free abelian by cyclic groups?? What happens for the direct product of two relatively hyperbolic groups.
\end{comment}
\ack
J. O. would like to thank Chris Hruska for directing the research of this topic in the context of his Ph.D. dissertation. Thanks also to Moon Duchin and Chris Mooney for originally formulating the question which ultimately led to this paper. W. Y. is supported by the Chinese grant ``The Recruitment Program of Global Experts". We are grateful to the referee for a careful reading and many useful comments and, in particular, for providing a simple argument to significantly shorten our proof of Lemma \ref{mainlayer}.

\section{Preliminaries}
Consider the Cayley graph $\Gx$ of $G$ with respect to $S$. Define $$B(1, n) = \{g \in G: d(1, g) \le n\}.$$ Let $S_n$ be the set of elements $g\in G$ such that $d(1, g)=n$. It will be useful to consider the spherical set in a subgroup $H$ in $G$. Define $$S_n(H) = H\cap S_n.$$

A parametrized path $p$ goes from $p_-$ to $p_+$ endowed with a natural order. For any two (parametrized) points $v, w \in p$, we denote by $[v, w]_p$ the segment between $v, w$ in $p$. As usual, $[v, w]$ denotes a (choice of) geodesic between $v, w$. Our path $p$ is often endowed with a length parameterization $p: [0, \len(p)] \to \Gx$.   

Let $p, q$ be two geodesics with the common initial endpoint $p_-=q_-$. A point $w\in q$ is called \textit{congruent} relative to $v\in p$  is satisfying that $d(v, p_-)=d(w, p_-)$.

Given a subset $X$ in $\Gx$,  the projection $\proj_X(v)$ of a point
$v$ to $X$ is the set of nearest points in $X$ to
$v$. For a subset $A \subset \Gx$, we define $\proj_X(A) = \cup_{a\in A
}\proj_X(a)$.

\subsection{Relative hyperbolicity and contracting property}
Given a finite collection of subgroups $\mathcal P$ in $G$, one can talk about the relative hyperbolicity of $G$ with respect to $\mathcal P$. From various points of view, the notion of relative hyperbolicity has been considered by many authors, cf. \cite{Gro}, \cite{Bow1}, \cite{Osin}, \cite{DruSapir}, and \cite{Ge1}, just to name a few.  These theories of relatively hyperbolic groups emphasize different aspects and are widely accepted to be equivalent for finitely generated groups. We refer the interested reader to \cite{Hru} and \cite{GePo3} for further discussions on their equivalence. 

In order to avoid heavy exposition, we only collect here necessary facts in the theory of relatively hyperbolic groups. Denote $\mathbb P =\{gP: g\in G, P\in \mathcal P\}$. Then $\mathbb P$ plays an important role in the geometry of $\Gx$, which has the following nice property.
\begin{defn}\label{contractdefn}
Let $\epsilon, D>0$. A subset $X$ is called \textit{$(\epsilon,
D)$-contracting} in $\Gx$ if the following holds
$$\diam{\proj_X(\gamma)} < D$$
for any geodesic $\gamma$ in $\Gx$ with $N_\epsilon(X) \cap \gamma =
\emptyset$.

A collection of $(\epsilon, D)$-contracting subsets is referred to as a
$(\epsilon, D)$-\textit{contracting system}. The constants $\epsilon, D$
will be often omitted, if no confusion happens. 
\end{defn}

We now recall some useful properties of contracting sets, and refer the reader to \cite{YANG6} for detailed discussions.
 
\begin{lem}[\cite{DruSapir}, \cite{GePo4}, \cite{YANG6}]\label{bipbpp}
Let $\GP$ be a relatively hyperbolic group. Then $\mathbb P$ is a contracting system with the following two
equivalent properties.
\begin{enumerate}
\item
\mbox{(bounded intersection property)} If for any $\epsilon >0$
there exists $R =\mathcal R(\epsilon)>0$ such that
$$
\diam{N_\epsilon(X) \cap N_\epsilon(X')} < R$$ for any two distinct
$X, X' \in \mathbb P$.
\item\mbox{(bounded projection property)} If there exists a finite number
$D>0$ such that
$$\diam{\proj_X(X')} < D$$
for any two distinct $X, X' \in \mathbb P$.
\end{enumerate}
\end{lem}
\begin{proof}
The contracting property was established in \cite[Proposition 8.2.4]{GePo4}. The property (1) was proved in \cite[Theorem 4.1]{DruSapir} and in \cite[Proposition 5.1.4]{GePo4}, and property (2) was in  \cite[Proposition 3.27]{GePo3}. The equivalence was shown in \cite[Lemma 2.3]{YANG6}.
\end{proof}

In the sequel, we will often invoke the function $\mathcal R$ without explicit mention of Lemma \ref{bipbpp}.

The following notion was introduced in \cite{Hru} by Hruska, and further elaborated on by Gerasimov-Potyagailo in \cite{GePo4}.
\begin{defn}
Fix $\epsilon, R>0$. Let $\gamma$ be a path in $\Gx$ and $v \in
\gamma$ a vertex. % such that $\gamma_-, \gamma_+ \notin B(v, R)$.
Given $X \in \mathbb P$, we say that $v$ is \textit{$(\epsilon,
R)$-deep} in $X$ if it holds that $\gamma \cap B(v, R) \subset
N_\epsilon(X).$ If $v$ is not $(\epsilon, R)$-deep in any $X \in
\mathbb P$, then $v$ is called an \textit{$(\epsilon, R)$-transition
point} of $\gamma$.
\end{defn}
%\begin{rem}
%We take a slightly different formulation from the ones \cite{Hru} and \cite{GePo4}:  the vertex $v$ does not necessarily have the property $\gamma_-, \gamma_+ \notin B(v, R)$. This simplifies our discussions
%\end{rem}

In what follows, there exists a uniform constant $\epsilon_0>0$ such that Lemmas  \ref{uniformtrans}, \ref{thintriangle} and \ref{neargeodesic} hold.

The first lemma is a consequence of contracting property of $\mathbb P$ (without the assumption of relative hyperbolicity). See \cite[Lemma 2.9]{YANG7} for a proof.
\begin{lem}\label{uniformtrans}
Let $p$ be a geodesic and a point $v \in p$ be $(\epsilon, R)$-deep in some $X\in \mathbb P$ for $\epsilon\ge \epsilon_0, R=\mathcal R(\epsilon)$. Denote by $x, y$ the entry and exit point of $p$ in $N_\epsilon(X)$ respectively.  Then $x, y$ are $(\epsilon_0, R)$-transition points.
\end{lem}

The following lemma could be derived using techniques in Section 8 in \cite{Hru} or it follows from the proof of Proposition 7.1.1 in \cite{GePo3} in terms of Floyd distance.
\begin{lem}\label{thintriangle}
Let $\epsilon\ge \epsilon_0, R=\mathcal R(\epsilon_0)$. There exists $D=D(\epsilon, R)$ with the following property.

Consider a geodesic triangle consisting of three geodesics $p, q, r$ in $\Gx$. Let $v$ be an $(\epsilon, R)$-transition point in $r$. Then there exists an $(\epsilon, R)$-transition point $w\in p\cup q$ such that $d(v, w) < D$.
\end{lem}

As a special case, we obtain the following result.
\begin{lem}\label{neargeodesic}
Let $\epsilon\ge \epsilon_0, R=\mathcal R(\epsilon_0)$. For any $r>0$, there exists $D=D(r)$ with the following property.

Let $p, q$ be two geodesics with $p_-=q_-$ and $d(p_+, q_+) \le r$. Consider an $(\epsilon, R)$-transition point $v \in p$. Then $d(v, q) \le D$.
\end{lem}
\begin{rem}
For convenience, it will be useful to take the congruent point $w\in q$ relative to $v\in p$ such that $d(v, w)\le D$ in the conclusion. In particular, $d(p_-, v)=d(p_-, w)$.
\end{rem}

\begin{comment}
A relative version of the thin-triangle
property in \cite{DruSapir} will also be key to our analysis in proof of Theorem \ref{mainthm}.

\begin{lem}\label{relthintri}
There exists a uniform constant $\delta>0$ such that for any geodesic triangle
$\Delta(v_0, v_1, v_2)$,  one of the following two cases holds
\begin{enumerate}
\item
there exists $o\in G$ such that $d(o, [v_i, v_{i+1}]) < \delta$ for
$i\mod3$,
\item
there exists a unique $gP \in \mathbb P$ such that $d(gP, [v_i,
v_{i+1}]) < \delta$ for $i \mod 3$.
\end{enumerate}
In the second case denote by $(v_i)_-, (v_i)_+$ the entry points of
$[v_i, v_{i-1}]$ and $[v_i, v_{i+1}]$ respectively in $N_\delta(gP)$
for $i \mod 3$. Then $d((v_i)_-, (v_i)_+) < \delta$.
\end{lem}
\end{comment}

\subsection{Exponential growth of balls}
We now consider a type of Poincar\`e series associated to a subset $A \subset G$ as follows,
$$
\PS{A} = \sum\limits_{a \in A} \exp(-s\cdot d(1, a)), s\ge 0.
$$
The \textit{critical exponent} $\nu_A$ of $\PS{A}$ is the limit
superior $$\nu_A=\limsup\limits_{n \to \infty} \frac{\log |B(1, n)
\cap A |}{n},$$ which can be thought of as the exponential growth rate of
$A$. Note that $\nu_G$ is the usual exponential rate $\g{G, S}$ of $G$
with respect to $S$. It is readily checked that $\PS{A}$ is convergent for $s > \nu_A$ and divergent for $s < \nu_A$.

Recall that a relatively hyperbolic group $G$ acts as a convergence group on its Bowditch boundary $\pG$, cf. \cite{Bow1}. Thus, every subgroup $H$ has a well-defined limit set $\Lambda(H) \subset \pG$, which consists of the set of accumulation points of all $H$-orbits in $\pG$. In \cite{YANG7}, the second-named author proves the following result.
\begin{lem}\label{convpara}\cite[Lemma 4.9]{YANG7}
Let $H$ be a subgroup in $G$ such that $\Lambda(H)$ is properly
contained in $\pG$. Then $\PS{H}$ is convergent at $s=\g G$. 
\end{lem}

Recall that every parabolic subgroup $P \in \mathcal P$ fixes a unique point in $\pG$, which coincides with the limit set $\Lambda(P)$. Lemma \ref{convpara} then applies and  the following result follows immediately.  

\begin{cor}\label{convergent}
The following series
$$
\sum\limits_{p \in P} \exp(-s\cdot d(1, p)) < \infty,
$$
or equivalently,
$$
\sum\limits_{n\ge 1} \exp(-sn) \cdot |S_n(P)| < \infty,
$$
for every $P \in \mathcal P$ and $s\ge \g G$.
\end{cor}

The following estimate is also important in the proof of Theorem \ref{mainthm}. The lower bound holds for any group, as a consequence of the sub-multiplicative inequality $|S_{n+m}|\le |S_n| |S_m|$.
\begin{lem}\label{expball}\cite[Theorem 1.8]{YANG7}
Let $G$ be a relatively hyperbolic group with a finite generating set $S$. Then there exists $c>1$ such that the following holds
\begin{equation}\label{expgrowth}
\exp(n \g G) \le |S_n| \le c \cdot \exp(n \g G).
\end{equation}
for any $n \ge 1$.
\end{lem}

\section{Proof of Theorem \ref{mainthm}}
The proof is organized into two parts, of which the first is to decompose $S_n$ into the union of a sequence of $C_{R+i}$ sets, and then the second is to execute the calculation  $\sum d(x, y)$ following the decomposition. We begin with the definition of uniform constants used below.
\begin{constants}\label{Constants}
Recall that $\mathcal R$ is the function given by Lemma \ref{bipbpp}.
\begin{enumerate}
\item
Let $\epsilon>0$ satisfy Lemmas \ref{uniformtrans}, \ref{thintriangle} and \ref{neargeodesic}. Assume also that $\epsilon, D_0>0$ are the contracting constants for $\mathbb P$.
\item
Let $R_0=\mathcal R(\epsilon)$.
%\item
%Let $\delta>0$ be given by Lemma \ref{relthintri}.
\item
Let $D_1=D(0)$ given by Lemma \ref{neargeodesic}.  We also demand that $D_1$ satisfies Lemma \ref{thintriangle}.
\end{enumerate}
\end{constants}

\subsection{Defining $C_{R+i}$ sets}
Fix any number $0 < \rho < 1/2$. We consider the sphere $S_{\rho n}$ for $n\ge 1$. For simplicity, assume that $\rho n$ is an integer. We will divide $S_n$ into disjoint well-controlled subsets.

Choose $R>\max\{2R_0, \mathcal R(2D_1)\}$. Let $C_R$ be the set of elements $g\in S_n$ such that there exists a geodesic $\gamma_g=[1, g]$ such that $\gamma_g$ contains an $(\epsilon, R_0)$-transition point in the (closed) $R$-neighbourhood of $\gamma_g(\rho n)$.

We consider any $g \in S_n \setminus C_R$. By definition of $C_R$, any geodesic $\gamma$ between $1$ and $g$ will not contain an $(\epsilon, R_0)$-transition point in the $R$-neighbourhood of $\gamma({\rho n})$. That is to say, the segment $\gamma([\rho n-R, \rho n+R])$ is contained in some $N_\epsilon(X_\gamma)$ for some $X_\gamma\in \mathbb P$. We first claim the following.
\begin{claim}
$X_\gamma$ is independent of the choice of $\gamma$.
\end{claim}
\begin{proof}[Proof of Claim]
If not, we have that $\gamma, \gamma', X_\gamma, X_{\gamma'}$ satisfy the requirement as above. Note that $\gamma, \gamma'$ have the same endpoints.  Let $x_-, x_+$ be the entry and exit points of $\gamma$ in $N_\epsilon(X_\gamma)$ respectively. The points $y_-, y_+\in \gamma'$ are similarly defined for $X_{\gamma'}$. Thus by Lemma \ref{uniformtrans}, $x_-,\; x_+,\; y_-,\; y_+$ are $(\epsilon, R_0)$-transitional points. By Lemma \ref{neargeodesic}, it follows that $$d(x_-, \gamma'),\; d(x_+, \gamma'),\; d(y_-, \gamma),\; d(y_-, \gamma) \le D_1.$$ Clearly, by the Remark after Lemma \ref{neargeodesic}, we see that  $N_{2D_1}(X_\gamma) \cap N_{2D_1}(X_{\gamma'})$ has diameter at least $2R\ge \mathcal R(2D_1)$. This implies $X_\gamma=X_{\gamma'}$ by bounded intersection property of $\mathbb P$.
\end{proof}

Thus, in what follows, we omit the index $\gamma$ in $X_\gamma$.

Let $z$ be the entry point of $\gamma$ in $N_\epsilon(X)$. By Lemma \ref{uniformtrans}, $z$ is an $(\epsilon, R_0)$-transition point in $\gamma$.
We observe that such $z$ lies in a uniformly bounded ball.
 
\begin{lem}\label{projcenter}
For any $g\in S_n \setminus C_R$, there exists a point $x \in X$ such that any geodesic $\gamma=[1, g]$ satisfies $d(x, z) \le D_0+\epsilon$. In particular, the set of $z \in \gamma$ for all possible $\gamma=[1, g]$ is uniformly bounded.
\end{lem}
\begin{proof}
Let $x\in X$ be a projection point of $1$ to $X$. By the contracting property of $X$, we see that $d(z, x) \le D_0+\epsilon$.
\end{proof}

We subdivide $S_n \setminus C_R$ and define a sequence of subsets as follows. For $i\ge 1$, define $C_{R+i}$ to be the set of elements $g$ in $S_n \setminus C_R$ where the point $z \in \gamma$ defined as above is nearest to $1$ among all $\gamma=[1, g]$ and has an exact distance $(R+i)$ to $\gamma(\rho n)$. We require that $R+i \le \rho n$ for obvious reasons.

We note the following fact. %Its proof goes in a similar way as that of the above Claim. We leave it to the interested reader.
\begin{lem}
$C_{R+i}\cap C_{R+j}=\emptyset$ for $i \ne j$.
\end{lem}

By the above discussion, we have the following disjoint union for $S_n$,  $$(\cup_{i\ge 1} C_{R+i}) \cup C_R = S_n.$$

Recall that $\mathcal P=\{P_k: 1\le k\le m\}$ is a finite set.
The following estimate is crux in the remaining argument, saying that $C_R$ occupies the major part of $S_n$ for sufficiently large $R\gg 0$.

\begin{lem}\label{crux}
\begin{comment}$$
\sum_{i\ge 0} \frac{|S_{n-\rho n}|\cdot |S_{r+i}(P_i)|}{|S_n|} \to 0,\;r\to \infty.
$$
By absolute convergence, we also have: [MAYBE WRONG]
$$
\sum_{i\ge 0} \frac{|S_{n-\rho n}|^2\cdot |S_{r+i}(P_i)|^2}{|S_n|^2} \to 0,\;r\to \infty.
$$
Furthermore,
\end{comment}
For any $\varepsilon>0$, there exists $R_1>0$ with the following property.  Let $R \ge R_1$ and $n\ge 1$ such that $\rho n >R$. Then the following holds
$$
\sum_{i\ge 1}  |C_{R+i}|/|S_n| \le \varepsilon,
$$
where $i\le \rho n -R$. %$R+i \le \rho n \le n- R- i$.
\end{lem}
\begin{proof}
By definition of $C_{R+i}$, for any $g\in C_{R+i}$, there exists a geodesic $\gamma_g=[1, g]$ such that $\gamma_g([\rho n- R- i, \rho n]) \subset N_\epsilon(X)$ for some $X \in \mathbb P$.  It then follows that
$$
|C_{R+i}| \le \sum_{1\le k \le m} |S_{\rho n-R-i}| \cdot |B(1, \epsilon)|  \cdot |S_{R+i+2\epsilon}(P_k)| \cdot |B(1, \epsilon)|\cdot |S_{n-\rho n}|,
$$
where $R+i \le \rho n$. Note that $|S_{R+i+2\epsilon}(P_k)| \le |S_{R+i}(P_k)| \cdot |S_{2\epsilon}|$ for $1\le k \le m$.
\begin{comment}
For each $i$, there exists a unique $g_iP_i \in \mathbb P$ such that
$|C_{r+i}| \le |S_{r+i}(P_i)| \cdot |S_{n-\rho n}|$.
Note that different $i$ give different $g_iP_i$.
\end{comment}
By Corollary \ref{convergent}, the following series 
\begin{equation}\label{convparaequ}
\sum_{i\ge 1} |S_{R+i}(P)|\cdot \exp(-\g G (R+i)) < \infty,
\end{equation}
is convergent for each $P \in \mathcal P$.
The conclusion then follows as a combination of the estimate (\ref{expgrowth}) of Lemma \ref{expball} and the convergent series (\ref{convparaequ}).
\end{proof}

\subsection{Calculating the sum $\sum d(x, y)$}
We first calculate the sum $\sum d(x, y)$, where $y$ lies in $C_R$.

Denote $F=B(1, 2D_1)$ and $D_{n, R} = 2(n-\rho n-R -D_1)>0$ for later use.
\begin{lem}\label{mainlayer}
%$$\sum_{x \in S_n, y\in C_r} d(x, y) \ge |C_r| \cdot (|S_n|-|F|\cdot |S_{n-\rho n}| -\sum_{j, k\ge r}^{j+k \le \rho n} |S_{j+k}(P)| \cdot |S_{n-\rho n -k}|) \cdot (2n-2\rho n).$$

%\begin{equation}\label{xysum1}
%\sum_{x \in S_n, y\in C_R} d(x, y) \ge |C_R| \cdot (|S_n|-|F|\cdot |S_{n-\rho n+R}| - \sum_{j\ge 1} |C_{R+j-\delta}|) \cdot D_{n, R}.
%\end{equation}
\begin{equation}\label{xysum1}
\sum_{x \in S_n, y\in C_R} d(x, y) \ge |C_R| \cdot (|S_n|-|F|\cdot |S_{n-\rho n+R}|) \cdot D_{n, R}.
\end{equation}
\end{lem}

\begin{proof}
For any $y \in C_R$, there exists a geodesic $\gamma_y=[1, y]$ such that $\gamma_y$ contains an $(\epsilon, R_0)$-transition point $z$ in the $R$-neighbourhood of $\gamma_y({\rho n})$. We can assume further that $z$ is nearest to $\gamma_y({\rho n-R})$ among all such $\gamma_y=[1, y]$. Then $d(\gamma_y({\rho n}), z) \le R$.

%The idea of this lemma is to find the appropriate subset of $x \in S_n$ such that
%\begin{equation}\label{largedist}
%d(x, y)\ge D_{n, R}.
%\end{equation}

We consider two sets $A, B$ of elements in $S_n$ separately. Let $A$ be the set of elements $x \in S_n$ such that $d(z, [1, x]) \le D_1$ for some $[1, x]$.  Thus it follows that 
\begin{equation}\label{excludeset1}
|A| \le |F|\cdot |S_{n-\rho n+R}|.
\end{equation}

Let $B=S_n\setminus A$. For any $x$ in $B$, we have $d(z, [1, x]) > D_1$, and then $d(z, [x, y]) \le D_1$ by Lemma \ref{thintriangle}. Observe that $d(x, z)  \ge d(y, z)$. Indeed, if $d(x, z) < d(y, z)$, then $d(1, x) \le d(x, z) + d(z, 1) < d(y, z) + d(z, 1) = d(1, y) $. This is a contradiction, since $x, y\in S_n$.

Let $w \in [x, y]$ such that $d(z, w) \le D_1$.  Note that $d(z, y) \ge n-\rho n -R$.  Thus, $\min\{d(y, w), d(x, w)\} \ge n-\rho n -R -D_1$. Therefore, the inequality (\ref{xysum1}) holds. 
\end{proof}

We now estimate the sum $\sum d(x, y)$ where $y \in C_{R+i}, i \ge 1$. The same proof as Lemma \ref{mainlayer} for the case $i=0$ proves the following.

\begin{lem}\label{otherlayer}
For each $i \ge 1$ with $R+i \le \rho n$, we have
\begin{equation}\label{xysum2}
\begin{array}{ll}
\sum_{x \in S_n, y\in C_{R+i}} d(x, y) \ge 
|C_{R+i}| \cdot (|S_n|-|F|\cdot |S_{n-\rho n+R+i}|)\cdot D_{n, R}.
\end{array}
\end{equation}
\end{lem}

\begin{comment}
Hence, similarly as above, we can assume that $\gamma_x|_{\rho n-R-i}$ is $(\epsilon, R)$-deep in $gP$. In this case, we have to exclude these set of elements $x\in S_n$, whose cardinality is estimated as follows.
$$
\sum_{j\ge r}^{j+i+r\le \rho n} (|S_{\rho n-j}|????) |F_1| \cdot |S_{j+i+r}(P)| \cdot |S_{n-\rho n}|.
$$
Note that $|S_{n-\rho n+R+i}| \le |S_{n-\rho n}| \cdot |S_{R+i}|$. In total, we need exclude at most
$$|F|\cdot |S_{R+i}(P_i)|\cdot |S_{n-\rho n}| + \sum_{j\ge 0} |C_{R+j}|.$$
Note that $R+i \le \rho n$.
\end{comment}

We are ready to finish the proof of Theorem \ref{mainthm}.
Combinning all of above inequalities in Lemmas \ref{mainlayer}, \ref{otherlayer} and \ref{crux}, we obtain the following
\begin{equation}\label{xysum}
\begin{array}{ll}
&\sum_{x, y\in S_n} d(x, y) \\
&\\
= &\sum_{i\ge 0}\sum_{x \in S_n, y\in C_{R+i}} d(x, y) \\
&\\
\ge & \sum_{i\ge 0} |C_{R+i}| \cdot (|S_n|-|F|\cdot |S_{n-\rho n+R+i}|)\cdot D_{n, R}\\
&\\
\ge & (|S_n|^2 - \sum_{i\ge 0} |F|\cdot  |C_{R+i}| \cdot |S_{n-\rho n+R+i}|) \cdot D_{n, R}.
\end{array}
\end{equation}
Therefore,  
$$
\frac{1}{|S_n|^2} \sum_{x, y\in S_n} \frac{d(x, y)}{n}\ge 2(1- \theta(n, R))(1-\rho - \frac{R+D_1}{n}).
$$
where 
$$
\theta(n, R) = (\sum_{i\ge 0} |F|\cdot  |C_{R+i}| \cdot |S_{n-\rho n+R+i}|)/|S_n|^2.
$$
Observe that 
\begin{lem}\label{ThetanR}
For any $\varepsilon>0$, there exists $R_1>0$ with the following property.
Let $R\ge R_1$ and $n \ge 1$ such that $\rho n  \ge R+ R_1$. Then
$\theta(n, R) \le \varepsilon$. 
\end{lem}
\begin{proof} 
We first consider the sum with $i=0$. Note that $C_{R} \subset S_n$. By Lemma \ref{expball}, there exists a uniform constant $\kappa>0$ such that 
$$
\frac{|F|\cdot |C_{R}| \cdot |S_{n-\rho n+R}|}{|S_n|^2} \le \frac{\kappa}{\exp(\g G(\rho n- R))}.
$$
Choose $R_1>0$ such that  $\kappa/\exp(\g G R_1) \le \epsilon/2$.

Consider now the sum with $i>0$. By Lemma \ref{crux}, we choose also $R_1>0$ such that $$
\sum_{i\ge 1}  |F|\cdot |C_{R+i}|/|S_n| \le \varepsilon/2,
$$
for $R>R_1$ and $i \le \rho n - R$. This clearly concludes the proof of the lemma.
\end{proof}
Thus, for any $\varepsilon>0$, we choose $R>0$ and let $n\to \infty$ to get $E(G, S) \ge 2(1-\varepsilon)(1-\rho)$.
As $\varepsilon, \rho$ are arbitrary, we then obtain that $E(G, S) =2$. 

The proof of Theorem \ref{mainthm} is complete.

\section{Statistical hyperbolicity of direct products}
This section is devoted to the proof of Theorem \ref{directproduct}. The outline is almostly the same as the proof of Lemma 5 (Annulus lemma) in \cite{DLM}, which is only sketched there. We provide here the details since we considered one relatively hyperbolic factor in $G\times H$, and our estimates in the proof of Theorem \ref{mainthm} are much more involved. 

We consider the direct product $G\times H$ with a split generating set $S$. Let $d$ be the word metric on $G\times H$ with respect to $S$. 

Denote $S_G=S\cap G$ and $S_H=S\cap H$. Then $S_G$ and $S_H$ generate $G$ and $H$ respectively. Recall that $S_n(X)$ denotes the part of the sphere $S_n$ in $X \subset G\times H$. Since $S$ is split,  it would be helpful to have in mind that $d(1, (g, h))=d_{S_G}(1, g) + d_{S_H}(1, h)$ for any $(g, h) \in G\times H$. Thus the sphere $S_n=S_n(G\times H) $ in $G\times H$ can be decomposed as in the following way,
$$
S_n = \cup_{0\le i\le n} S_i(G) \times S_{n-i}(H).
$$ 
Note that $S_i(G)$ coincides with the sphere of radius $i$ in the Cayley graph $\mathscr G(G, S_G)$ of $G$ with respect to $S_G$.   

\begin{lem}\label{nullset}
For any fixed $0 < t < 1$,  the following holds
\begin{equation}\label{almostall}  
\frac{|\cup_{0 \le i\le tn} S_i(G) \times S_{n-i}(H)|}{|S_n|} \to 0, \; n\to \infty
\end{equation}

\end{lem}
\begin{proof}
We use $\prec$ and $\asymp$ to denote the inequality and equality respectively, up to a computable multiplicative constant. Note that there exists $\nu_{G, S_G}> \nu >\nu_{H, S_H}$ such that $|S_i(H)| \prec \exp(\nu i)$ for all $i>0$. For simplicity, denote $\nu_G=\nu_{G, S_G}$. 

Since $G$ is relatively hyperbolic, it follows by Lemma \ref{expball} that $|S_i(G)|\asymp \exp(i \nu_{G})$ for $i \ge 0$. Observe that  
$$
\begin{array}{ll} 
\frac{|\cup_{0 \le i\le tn} S_i(G) \times S_{n-i}(H)|}{|\cup_{tn \le i\le n} S_i(G) \times S_{n-i}(H)|} & \prec \frac{\sum_{0 \le i\le tn} \exp(i\nu_G) \exp((n-i)\nu)}{\sum_{tn \le i\le n} \exp(i \nu_G)} \\
&\\
& \prec  \frac{\exp(tn(\nu_G-\nu))}{\exp((n+1)(\nu_G-\nu))(1-\exp((tn-n)\nu_G)} \\
&\\
& \prec  \frac{\exp((tn-n-1)(\nu_G-\nu))}{1-\exp((tn-n)\nu_G)},
\end{array}
$$ which tends to $0$ as $n\to \infty$ for any fixed $0 < t < 1$.
\end{proof}

We now proceed as in Section 3, and indicate the necessary changes. Fix any number $0 < \rho < 1/2$.  Assume that $1> t >  \rho$. 

We consider the annular-like set $A_{tn, n} :=\cup_{tn \le i\le n} S_i(G) \times S_{n-i}(H)$. By Lemma \ref{nullset}, we know that 
\begin{equation}\label{almostall2}
|A_{tn, n}|/|S_n|\to 1
\end{equation}
as $n\to \infty$.

Choose $R>\max\{2R_0, \mathcal R(2D_1)\}$, where $R_0$ and $D_1$ are given by Constants \ref{Constants}. We define $C_{R+i}$ sets in $A_{tn, n}$ for $i\ge 0$ as in Section 3, where $S_n$ is replaced by $A_{tn, n}$. %We precise the modification as follows. 

%Consider any $(g, h) \in A_{tn, n} \setminus C_R$, where $g\in S_j(G), h\in S_{n-j}(H)$ for $tn\le j \le n$. 

Let $C_R$ be the set of elements $(g, h) \in A_{tn, n}$ such that there exists a geodesic $\gamma_g=[1, g]$ in the Cayley graph $\mathscr G(G, S_G)$ of $G$ such that $\gamma_g$ contains an $(\epsilon, R_0)$-transition point in the (closed) $R$-neighbourhood of $\gamma_g(\rho n)$.

%By definition of $C_R$, any geodesic $\gamma$ between $1$ and $g$ will not contain an $(\epsilon, R_0)$-transition point in the $R$-neighbourhood of $\gamma({\rho n})$. That is to say, the segment $\gamma([\rho n-R, \rho n+R])$ is contained in some $N_\epsilon(X_\gamma)$ for some $X_\gamma\in \mathbb P$.

We continue to subdivide $A_{tn, n} \setminus C_R$. For $i\ge 1$, define $C_{R+i}$ to be the set of elements $(g, h)$ in $A_{tn, n} \setminus C_R$ where the point $z \in \gamma$ defined in Section 3.1 is nearest to $1$ among all $\gamma=[1, g]$ in $\mathscr G(G, S_G)$ and has an exact distance $(R+i)$ to $\gamma(\rho n)$. Therefore, $A_{tn, n}=\cup_{i\ge 0} C_{R+i}$ as a disjoint union.  

We prove an analogue of Lemma \ref{crux}.
\begin{lem}\label{crux2}
For any $\varepsilon>0$, there exist $R_1>0$ with the following property.  Let $R \ge R_1$ and $n\ge 1$ such that $\rho n > R$. Then the following holds
$$
\sum_{i\ge 1}  |C_{R+i}|/|A_{tn, n}| \le \varepsilon,
$$
where $i \le \rho n -R$.
\end{lem}
\begin{proof}
By definition of $C_{R+i}$, for any $(g, h)\in C_{R+i}$, there exists a geodesic $\gamma_g=[1, g]$ such that $\gamma_g([\rho n- R- i, \rho n]) \subset N_\epsilon(X)$ for some $X \in \mathbb P$. It then follows that
$$
\begin{array}{ll}
|C_{R+i}| \le &  \sum_{tn \le j \le n} \sum_{1\le k \le m} |S_{\rho n-R-i}(G)|  \cdot\\
&\\
&   |S_{R+i}(P_k)|\cdot |B(1, 4\epsilon)|  \cdot |S_{j- \rho n}(G)|\cdot |S_{n-j}(H)|
\end{array}
$$
where $R+i \le \rho n$ and $B(1, \epsilon)$ should be understood as the ball in the Cayley graph $\mathscr G(G, S_G)$. 

By Lemma \ref{expball} there exists $c>1$ such that $\exp(l \g G ) \le |S_{l}(G)| \le c\cdot \exp(l\g G )$ for any $l\ge 1$. Thus we obtain that 
$$
|C_{R+i}| \le c^2 |B(1, 2\epsilon)| \cdot \sum_{tn \le j \le n} \sum_{1\le k \le m} \exp(\nu_G(j -R-i))  \cdot |S_{R+i}(P_k)| \cdot |S_{n-j}(H)|.
$$
On the other hand, 
$$|A_{tn, n}| =\sum_{tn \le j\le n} |S_{j}(G)| \cdot |S_{n-j}(H)| \ge \sum_{tn \le j\le n} \exp(j \g G) \cdot |S_{n-j}(H)|.$$ 
 
Similarly as in the proof of Lemma \ref{crux}, the conclusion follows as a consequence of the convergent series given by Corollary \ref{convergent}.
\end{proof}
 
Denote $F=B(1, 2D_1)$, which is the ball in $\mathscr G(G, S_G)$, and $D_{n, R} = 2(tn-\rho n-R -D_1)>0$.  We proceed as in Lemma \ref{mainlayer} to get the following

\begin{lem}\label{productlayer}
For each $i \ge 0$ with $R+i \le \rho n$,  
\begin{equation}\label{xysum3}
\begin{array}{ll}
\sum_{x \in A_{tn, n}, y\in C_{R+i}} d(x, y) & \ge  D_{n, R} \cdot |C_{R+i}| \cdot  (|A_{tn, n}| - \\
& \sum_{tn \le j \le n} |F|\cdot |S_{j-\rho n+R+i}(G)| \cdot |S_{n-j}(H)|).
\end{array}
\end{equation}
\end{lem}

\begin{proof}
We sketch the arguments in the proof of Lemma \ref{mainlayer} with necessary changes.

For any $y=(g_y, h_y) \in C_{R+i}$, there exists a geodesic $\gamma_y=[1, g_y]$ in $\mathscr G(G, S_G)$ such that $\gamma_y$ contains an $(\epsilon, R_0)$-transition point $z$ in the $(R+i)$-neighbourhood of $\gamma_y({\rho n})$.  Then $d(\gamma_y({\rho n}), z) \le R+i$.

Let $A$ be the set of elements $x=(g_x, h_x) \in A_{tn, n}$ such that $d_{S_G}(z, [1, g_x]) \le D_1$.  Thus the cardinality of $A$ is at most
\begin{equation}\label{excludeset1}
|F|\cdot \sum_{tn \le j \le n} |S_{j - \rho n+R+i}(G)| \cdot |S_{n-j}(H)|.
\end{equation}

Let $B=A_{tn, n} \setminus A$. For any $x=(g_x, h_x)$ in $B$, we have $d_{S_G}(z, [1, g_x]) > D_1$, and then $d_{S_G}(z, [g_x, g_y]) \le D_1$ by Lemma \ref{thintriangle}. %Here the metric $d$ could be understood either as the word metric $d_S$ on $G \times H$ or as the word metric $d_{S_G}$ on $G$, because they are equal on $G$.

Let $w \in [g_x, g_y]$ such that $d(z, w) \le D_1$.  Then an argument as in Lemma \ref{mainlayer} proves that $\min\{d_{S_G}(g_y, w), d_{S_G}(g_x, w)\} \ge tn-\rho n -R -D_1$. The inequality (\ref{xysum3}) then holds. 
\end{proof}
 
So we have the sum estimate as follows,
$$
\begin{array}{ll}
&\sum_{x, y\in A_{tn, n}} d(x, y) \\
&\\
= &\sum_{i\ge 0}\sum_{x \in A_{tn, n}, y\in C_{R+i}} d(x, y) \\
&\\
\ge & (|A_{t n, n}|^2 - \sum_{i\ge 0} \sum_{tn \le j \le n} |F| \cdot  |C_{R+i}|  \cdot |S_{j-\rho n+R+i}(G)| \cdot |S_{n-j}(H)| ) \cdot D_{n, R}.
\end{array}
$$
Therefore,  
$$
\frac{1}{|S_n|^2} \sum_{x, y\in S_n} \frac{d(x, y)}{n} \ge \frac{1}{|S_n|^2} \sum_{x, y\in A_{tn, n}} \frac{d(x, y)}{n} \ge 2 \frac{|A_{tn, n}|^2}{|S_n|^2} (1- \theta(n, R))(t-\rho - \frac{R+D_1}{n}).
$$
where 
$$
\theta(n, R) = (\sum_{i\ge 0} \sum_{tn \le j \le n} |F| \cdot  |C_{R+i}| \cdot |S_{j-\rho n+R+i}(G)| \cdot |S_{n-j}(H)|)/|A_{tn, n}|^2.
$$

We can prove a similar statement for $\theta(n, R)$ by the same reasoning as in Lemma \ref{ThetanR}. 
\begin{lem}\label{ThetanR2}
For any $\varepsilon>0$, there exists $R_1>0$ with the following property.
Let $R\ge R_1$ and $n \ge 1$ such that $\rho n  \ge R+ R_1$. Then
$\theta(n, R) \le \varepsilon$. 
\end{lem}
\begin{proof}[Sketch of Proof] 
Recall that $A_{tn, n} =\cup_{tn \le j\le n} S_j(G) \times S_{n-j}(H)$.  Note that  $C_{R} \subset S_n$ and by (\ref{almostall2}),  $|S_n|/|A_{tn, n}| \to 1$ as $n\to \infty$. For the sum with $i=0$, it suffices to estimate the following by Lemma \ref{expball},  
$$
\frac{\sum_{tn \le j \le n}  |S_{j-\rho n+R}(G)|\cdot |S_{n-j}(H)|}{|A_{tn, n}|}  \prec \frac{1}{\exp(\g G(\rho n- R))} \sum_{tn \le j \le n} \frac{1}{\exp(\g G(n-j))},
$$
which tends to $0$, as $(\rho n- R)\to \infty$.  

For the sum with $i\ge 1$, since
$
\sum_{tn \le j \le n} |S_{j-\rho n+R+i}(G)| \cdot |S_{n-j}(H)| \le |A_{tn, n}|,
$
we have by Lemma \ref{crux2},
$$
\sum_{i\ge 1}^{i\le \rho n- R}  |C_{R+i}|/|A_{tn, n}| \to 0, 
$$
as $R\to \infty$.

The proof of the lemma follows easily from the above estimates.
\end{proof}

Finally, for any $\varepsilon>0$, we choose $R>0$ and let $n\to \infty$ to get $E(G, S) \ge 2(1-\varepsilon)(t-\rho)$.
As $\varepsilon, t, \rho$ are arbitrary, we then obtain that $E(G, S) =2$. 

% For alignments use AmS-LaTeX constructions not \eqnarray.

% For acknowledgements use \ack or \acks immediately before the references

\bibliographystyle{amsplain}
 \bibliography{bibliography}

\end{document}